\documentclass[11pt, reqno]{amsart}
\usepackage{amsmath, amsthm, amscd, amsfonts, amssymb, graphicx, color}

\newtheorem{theorem}{Theorem}[section]
\newtheorem{cor}[theorem]{Corollary}
\newtheorem{lem}[theorem]{Lemma}

\theoremstyle{definition}
\newtheorem{defn}[theorem]{Definition}
\newtheorem{exam}[theorem]{Example}

\numberwithin{equation}{section}

\newcommand{\A}{\mathcal{A}}

\newcommand{\LL}{\mathcal{L}}

\newcommand{\B}{\mathcal{B}}

\begin{document}
\title[]{Ultrafilters on measurable semigroups}
\author[]{A. Pashapournia\\
M. Akbari Tootkaboni$^\dag$\\
D. Ebrahimbagha}
\address{A. Pashapournia}
\address{Department of Mathematics\\
Faculty of Sciences\\
Islamiic Azad University\\
 Central Tehran Branch.}
\address{M. Akbari Tootkaboni}
\address{Department of pure Mathematics\\
Faculty of Mathematical Sciences\\
University of Guilan\\
Rasht-Iran.}
\address{ E-mail: tootkaboni@guilan.ac.ir}
\address{D. Ebrahimbagha}
\address{Department of Mathematics\\
Faculty of Sciences\\
Islamiic Azad University\\
 Central Tehran Branch.}
\address{E-mail: e-bagha@yahoo.com}
\keywords{Copmpactification, Measurable Space, Ultrafilter, Semigroup. 
 $^\dag$Corresponding Arthur }
\begin{abstract}
Let $(S,\cdot)$ be a semigroup and $\mathfrak{m}$ be a $\sigma$-algebra on $S$. We say $(S,\cdot,\mathfrak{m})$ is a measurable semigroup if $\pi:S\times S\longrightarrow S$ by $\pi(x,y)=x\cdot y$ is a measurable function.
In this paper , we consider to $\mathfrak{m}^\beta$ as the collection of all ultrafilters on $\mathfrak{m}$.
 We show that $\mathfrak{m}^\beta$ is a compact right topological semigroup respect to generated topology by $\sigma$-algebra $\mathfrak{m}$ on $\mathfrak{m}^\beta$. Also we study some elementary properties of the algebraic structure of $\mathfrak{m}^\beta$.
\end{abstract}

\maketitle
\section{\textbf{Introduction}}

A semigroup $S$ which is
also a  Hausdorff topological space is called a semitopological semigroup, if  for each $s\in S$, $\lambda_{s} :S\rightarrow S$ and $r_{s} :S\rightarrow S$ are
continuous, where for each $x\in S$, $\lambda_{s}(x)=sx$ and
$r_{s}(x)=xs$. If  just $r_s$, for each $s\in S$, is continuous, $ S$ is called a right topological semigroup.

A pair $(\psi,X)$ is a semigroup compactification of $S$ if $X$ is a Hausdorff compact right topological
semigroup and $\psi:S\rightarrow X$ is  continuous homomorphism with dense image
such that for all $s\in S$, the mapping
$x\mapsto\psi(s)x:X\rightarrow X$ is continuous. For more details
see \cite[Section 2]{Analyson}.

All semigroup compactifications of a semitopological semigroup as a
collection of $z$-ultrafilters or $e$-ultrafilters have been described, see \cite{Alaste1}, \cite{Alaste2}, \cite{Akbarii} and \cite{Akbari} for more details and some applications. This approach gives us a new light on studying this
kind of compactifications. It seems that the methods presented in
\cite{Alaste1}, \cite{Alaste2}, \cite{Akbarii} and
 \cite{Akbari} can serve as a valuable tool in the study of semigroup compactifications and also of topological compactifications. But this methods depend on concepts of topology on semigroups. For a clear example, the Stone-$\check{C}$ech compactification of real additive numbers equipped to the natural topology is not semigroup compactification.

Let $(S,\cdot)$ be a semitopological semigroup. In Preliminary for a nonempty set $X$, we define ultrafiler on a $\sigma$-algebra $\mathfrak{m}$, is called $\mathfrak{m}$-ultrafilter. Also, we study $\mathfrak{m}^\beta$ as a collection of all $\mathfrak{m}$-ultrafilters on $\mathfrak{m}$ respect to $\sigma$-algebra generated  by $\{\widehat{A}:A\in\mathfrak{m}\}$, where $\widehat{A}=\{p\in\mathfrak{m}^\beta:A\in p\}$.

In Section 3, for two measurable spaces $(X,\mathfrak{m})$ and $(Y,\mathfrak{n})$ we show that if $f:(X,\mathfrak{m})\rightarrow(Y,\mathfrak{n})$ is measurable function then there exists a unique measurable extension of $f$.

In Section 4, we define measurable semigroup and we extend $"\cdot"$ naturally to $"*"$ on $\mathfrak{m}^\beta$. Also, some elementary algebraic properties of $(\mathfrak{m}^\beta,*)$ as extension of measurable semigroup $(S,\cdot,\mathfrak{m})$ is stated.

In Section 5, we concentrate on the Lebesgue measurable subsets od $S=(0,+\infty)$.

\section{\textbf{Preliminary}}

Let $X$ be a nonempty set and $\mathfrak{m}$ be an infinite $\sigma$-algebra on $X$. Then $(X,\mathfrak{m})$ is called measurable space. We say that $\mathfrak{m}$ separate $X$ if for each $x,y\in X$ there exist $A,B\in\mathfrak{m}$ such that $x\in A$, $y\in B$ and $A\cap B=\emptyset$.
If $\mu$ is a measure on $\mathfrak{m}$, then $(X,\mathfrak{m}, \mu)$ is a measure space. The outer measure $\mu^*$ can be defined
 for every nonnegative measure $\mu$. The collection of $\mu^*-$measurable sets is denoted by $\mathfrak{m}_{\mu}$. Also $\mu^*$ is a measure on
$\mathfrak{m}_{\mu}$,  $(X,\mathfrak{m}_{\mu}, \mu^*)$ is a Lebesgue extension or the Lebesgue completion of the  measure $\mu$, and $\mathfrak{m}\subseteq\mathfrak{m}_{\mu}$, see Theorem 1.5.6 in \cite{Bogachev}.

 Let $(X,\mathfrak{m}_1)$ and $(X,\mathfrak{m}_2)$ be measurable spaces. The function $ f: X\rightarrow Y$ is an
$(\mathfrak{m}_1,\mathfrak{m}_2)$-measurable function if  $f^{-1}(A)\in \mathfrak{m}_1$ for every $A\in \mathfrak{m}_2$.

Let $(Y, \tau)$ be a topological space. The $\sigma$-algebra generated by $\tau$ is called the Borel $\sigma$-algebra and is denoted by $B_\tau$.
The function  $f: (X,\mathfrak{m})\rightarrow (Y,\tau)$ is an
$\mathfrak{m}$-measurable function if  $f^{-1}(A)\in \mathfrak{m}$ for every $A\in \tau$.

Let $(X,\A)$ and $(Y,\B)$ be two measurable spaces. For $X\times Y$, we consider the collection of all sets of the from $A\times B$ , where $A\in \A$ and $B\in \B$, called measurable rectangle. The $\sigma$-algebra generated by all measurable rectangles is called product of the $\sigma$-algebras $\A $ and $\B $, and is denoted by $\A\bigotimes \B$. Let $ \mu $ and $\nu $ be measures on $\A $ and $\B $, respectively. Define $\mu \times \nu (A\times B) =\mu(A) \nu(B) $ for each measurable rectangles $A\times B$.

\begin{theorem}
 The set function $\mu\times \nu $ extends to a countably additive measure, denoted by $\mu\bigotimes \nu $, on $\A\bigotimes \B $.
 \end{theorem}

Proof: See Theorem 3.3.1 in \cite{Bogachev}.
\begin{theorem}
Let $(X_1,\mathcal{A}_1)$ and $(X_2,\mathcal{A}_2)$ be measurable spaces. Then the set
$A_{x_1}=\{x_2\in X_2:(x_1,x_2)\in A\}\in \mathcal{A}_2$ for every $A\in\mathcal{A}_1\otimes\mathcal{A}_2$ and $x_1\in X_1$.
Also the set
$A^{x_2}=\{x_1\in X_1:(x_1,x_2)\in A\}\in \mathcal{A}_1$ for every $A\in\mathcal{A}_1\otimes\mathcal{A}_2$ and $x_2\in X_2$.
\end{theorem}
\begin{proof}
See Proposition 3.3.2 in \cite{Bogachev}.
\end{proof}

\begin{defn}
Let $X$ be a nonempty set and $\mathfrak{m}$ be an infinite $\sigma$-algebra on $X$. We say $p\subseteq \mathfrak{m}$ is an $\mathfrak{m}$-filter if

$(a)$ $\emptyset\notin p$ and $X\in p$.

$(b)$ If $A,B\in  \mathfrak{m}$, $A\in p$, and $A\subseteq B$, then $B\in  p$.

$(c)$ If $A,B\in p$, then $A\cap B\in p$.
\end{defn}

An $\mathfrak{m}$-ultrafilter is an $\mathfrak{m}$-filter which is not properly contained in
any other $\mathfrak{m}$-filter. Zorn’s Lemma guarantees that there exist $\mathfrak{m}^\beta$-ultrafilters.
\begin{lem}\label{L1}
 Let $p$ be an $\mathfrak{m}$-filter and $A\in \mathfrak{m}$. Either

$(a)$ there is some $B\in p$ such that $A\cap B=\emptyset$ or

$(b)$ $\{ C\in\mathfrak{m}:\ there \ is \ some\ B\in p \ with\ A\cap B\subseteq C\}$ is an $\mathfrak{m}$-filter.
\end{lem}
\begin{proof}
If there is $(a)$, then the statement is obvious.

If $(a)$ is false, then $A\cap B\neq \emptyset$ for every $B\in p$. Let
\[
\mathcal{A}=\{C\in \mathfrak{m}:\ there \ is \ some\ B\in p \ with\ A\cap B\subseteq C\}.
\]
It is obvious that $A\in\mathcal{A}$, $\emptyset\notin \mathcal{A}$, and $X\in \mathcal{A}$. If $C\in \mathcal{A}$ and $C\subseteq C'\in\mathfrak{m}$, then $C'\in \mathcal{A}$. Also if $C,D\in\mathcal{A}$, then there exist $B_1, B_2\in p$ such that $B_1\cap A\subseteq C$ and $B_2\cap A\subseteq D$. So
$(B_1\cap B_2)\cap A\subseteq D\cap C$ and hence $B_1\cap B_2\in p$ Since $p$ is an $\mathfrak{m}$-filter. So $D\cap C\in\mathcal{A}$ and this implies
that $\mathcal{A}$ is an $\mathfrak{m}$-filter.
\end{proof}
\begin{theorem}\label{T2}
Let $\mathfrak{m}$ be an infinite $\sigma$-algebra on $X$ and $p\subseteq\mathfrak{m}$ be an $\mathfrak{m}$-filter. The following statements are
equivalent:

$(a)$ $p$ is an $\mathfrak{m}$-ultrafilter.

$(b)$ for all $A,B\in \mathfrak{m}$, if $A\cup B\in p$ then $ A\in p$ or $B\in p$.

$(c)$ for all $A\in \mathfrak{m}$ either $ A\in p$ or $A^c\in p$.

 $(d)$ for each $A\in \mathfrak{m}\setminus p$ there is some $B
\in p$ such that $A\cap B=\emptyset$.
\end{theorem}
\begin{proof}
(a) implies (b). Let $A, B\in \mathfrak{m}$ such that $A\cup B\in p$. If $B\notin p$ and $A\notin p$, then there are
$C, D\in p$ such that $A\cap C=\emptyset$ and by Lemma \ref{L1}, $B\cap D=\emptyset$. So
\begin{align*}
(A\cup B)\cap (D\cap C) &= (A\cap(D\cap C))\cup (B\cap(D\cap C))\\
& \subseteq (A\cap C)\cup(B\cap D)=\emptyset. \\
\end{align*}

Then $\emptyset\in p$, is contradiction.

(b) implies (c). It is obvious.

(c) implies (d). Let $A\in\mathfrak{m}$ such that $A\notin p$. Then $A^c\in p$ so $A\cap A^c=\emptyset$.

(d) implies (a). Let $p$ be an $\mathfrak{m}$-filter and  let $\mathcal{A}$ be an $\mathfrak{m}$-ultrafilter such that $p\subseteq\mathcal{A}$.
 Pick  $A\in \mathcal{A}\setminus p$. Then $A\notin p$ so there is a some $B\in p$ such that $A\cap B=\emptyset$.
 But $\emptyset=A\cap B\in \mathcal{A}$, a contradiction.
\end{proof}
\begin{defn}
Let $X$ be a nonempty set and $\mathfrak{m}$ be an infinite $\sigma$-algebra on $X$. For $x\in X$, we define
\[
\hat{x}=\{A\in \mathfrak{m}:x\in A\}.
\]
\end{defn}
and $\hat{x}$ is called $\mathfrak{m}$-principal ultrafilter.

 It is obvious that $p$ is the $\mathfrak{m}$-principal ultrafilter if and only if $\bigcap p\neq\emptyset$.
\begin{theorem}\label{T1}
Let $X$ be a nonempty set,  $\mathfrak{m}$ be an infinite $\sigma$-algebra on $X$, and $\mathcal{A}\subseteq\mathfrak{m}$ has the finite
intersection property. Then there is an $\mathfrak{m}$-ultrafilter $p$ such that $\mathcal{A}\subseteq p$.
\end{theorem}
\begin{proof}
 By Zorn's Lemma is obvious.
\end{proof}
\begin{defn}
Let $X$ be a nonempty set, let $\mathfrak{m}$ be an infinite $\sigma$-algebra on $X$, and let $\mathcal{R}$ be a nonempty set such that $\mathcal{R}\subseteq\mathfrak{m}$. We say that $\mathcal{R}$ is partition regular
if and only if whenever $A,B\in \mathfrak{m}$ and $A\cup B\in \mathcal{R}$, there exists $C\in \mathcal{R} $
such that $C\subseteq A$ or $C\subseteq B$.
\end{defn}
\begin{theorem}
Let $X$ be a nonempty set, let $\mathfrak{m}$ be an infinite $\sigma$-algebra on $X$, and let $\mathcal{R}$ be a nonempty set such that $\mathcal{R}\subseteq\mathfrak{m}$ and assume that $\emptyset\notin \mathcal{R}$.
Let $R^{\uparrow}=\{B\in \mathfrak{m}:A\subseteq B\mbox{ for some } A\in\mathcal{R}\}$.
The following statements
are equivalent:

$(a)$ $\mathcal{R}$ is partition regular.

$(b)$ Whenever $\mathcal{A}\subseteq\mathfrak{m}$ has the property that every finite nonempty subfamily of
$\mathcal{A}$ has an intersection which is in $R^{\uparrow}$, there is an $\mathfrak{m}$-ultrafilter $p\subseteq\mathfrak{m}$ such that
$\mathcal{A}\subseteq p\subseteq R^{\uparrow}$.

$(c)$ Whenever $A\in\mathcal{R}$, there is some $\mathfrak{m}$-ultrafilter $p$ such that $A\in p\subseteq R^{\uparrow}$.
\end{theorem}
\begin{proof}
(a) implies (b). Let $\mathcal{B}=\{A\in\mathfrak{m}:\forall\ B\in \mathcal{R}, A\cap B\neq \emptyset\}$ and note that
$\mathcal{B}\neq\emptyset$. We may assume that $\mathcal{A}\neq\emptyset$, since $\{X\}$ has the
hypothesized property. Let $\mathcal{C}=\mathcal{A}\cup\mathcal{B}$. We claim that $\mathcal{C}$ has the finite intersection
property. To see this it suffices (since $\mathcal{A}$ and $\mathcal{B}$ are nonempty) to let $\mathcal{F}\in P_f(\mathcal{A})$ and
 $\mathcal{G}\in P_f(\mathcal{B})$ and show that $\bigcap\mathcal{F}\cap\bigcap\mathcal{G}\neq\emptyset$. So suppose instead that we have such $\mathcal{F}$
and $\mathcal{G}$ with $\bigcap\mathcal{F}\cap\bigcap\mathcal{G}=\emptyset$.  Pick $B\in\mathcal{R}$ such that $B \subseteq\bigcap\mathcal{F}$.
Then $B\cap\bigcap\mathcal{G}=\emptyset$
and so $B=\bigcup_{A\in\mathcal{G}}(B\setminus A)$. Pick $A\in\mathcal{G}$ and $C\in\mathcal{R}$ such that $C\subseteq B\setminus A$. Then
$A\cap C=\emptyset$, contradicting the fact that $A\in\mathcal{B}$.

By Theorem \ref{T1}, there is an $\mathfrak{m}$-ultrafilter $p$ such that $\mathcal{C}\subseteq p$. Given $C\in p$,
$X\setminus C\notin\mathcal{B}$ (since $C\cap(X\setminus C)=\emptyset \notin p$). So pick some $B\in\mathcal{R}$ such that $B\cap(X\setminus C)=\emptyset$.
That is, $B\subseteq C$.

(b) implies (c). Let $\mathcal{A}=\{A\}$.

(c) implies (a). Let $\mathcal{F}$ be a finite set of elements of  $\mathfrak{m}$ with $\bigcup\mathcal{F}\in \mathcal{R}$ and let $p$ be an
$\mathfrak{m}$-ultrafilter such that $\bigcup\mathcal{F}\in p$ and for each $C\in p$ there is some $B\in\mathcal{R}$ such
that $B\subseteq C$. Pick by Theorem \ref{T2} some $A\in\mathcal{F}\cap p$.
\end{proof}
\begin{cor}\label{C1}
Let $(X,\mathfrak{m})$ be a measurable space. Let $\mathcal{A}\subseteq \mathfrak{m}$ be an arbitrary
family. If the intersection of every finite subfamily of $\mathcal{A}$ is infinite, then $\mathcal{A}$ is contained in an $\mathfrak{m}$-ultrafilter
all of whose members are infinite. More generally, if $\kappa$ is an infinite cardinal
and if the intersection of every finite subfamily of $\mathcal{A}$ has cardinality at least $\kappa$, then there exists
$\mathfrak{m}$-ultrafilter $p$ such that $|A|\geq\kappa$ for every $A\in p$.
\end{cor}
\begin{proof}
The proof is obvious.
\end{proof}

Let $(X,\mathfrak{m})$ be a measurable space. The collection of all $\mathfrak{m}$-ultrafilters
is denoted by $\mathfrak{m}^\beta$. For each $A\in\mathfrak{m}$, we define
\[
\widehat{A}=\{p\in\mathfrak{m}^\beta:A\in p\}.
\]
\begin{lem}\label{L2}
Let $X$ be a nonempty set, let $\mathfrak{m}$ be an infinite $\sigma$-algebra on $X$, and let $A,B\in \mathfrak{m}$.

$(a)$ $\widehat{A\cap B}=\widehat{A}\cap\widehat{B}$;

$(b)$ $\widehat{A\cup B}=\widehat{A}\cup\widehat{B}$;

$(c)$ $\widehat{X\setminus A}=\mathfrak{m}^{\beta}\setminus\widehat{A}$;

$(d)$ $\widehat{A}=\emptyset$ if and only if $A=\emptyset$;

$(e)$ $\widehat{A}=\mathfrak{m}^{\beta}$ if and only if $A=X$;

$(f)$ $\widehat{A}=\widehat{B}$ if and only if $A=B$.
\end{lem}
\begin{proof}
The proof is obvious.
\end{proof}
The collection $\{\widehat{A}:A\in\mathfrak{m}\}$ is a basis for topology on $\mathfrak{m}^\beta$.
\begin{theorem}\label{T3}
Let $X$ be a nonempty set, let $\mathfrak{m}$ be an infinite $\sigma$-algebra on $X$. Then $\mathfrak{m}^\beta$ is a compact and  Hausdorff space.
\end{theorem}
\begin{proof}
Suppose that $p$ and $q$ are distinct elements of $\mathfrak{m}^\beta$. If $A\in p\setminus q$, then
$X\setminus A\in q$. So $\widehat{A}$ and  $\widehat{X\setminus A}$ are disjoint open subsets of $\mathfrak{m}^\beta$ containing $p$ and $q$, respectively. Thus $\mathfrak{m}^\beta$ is Hausdorff.

 To show that $\mathfrak{m}^\beta$ is compact, we shall consider a family
$\mathcal{A} $ of sets of the form $\widehat{A}$ with the finite intersection property and show that $\mathcal{A}$ has a
nonempty intersection. Let $\mathcal{B}=\{A\in\mathfrak{m}:\widehat{A}\in\mathcal{A}\}$. If $F\in P_f(\mathcal{B})$, then there is
some $p\in \bigcap_{A\in \mathcal{F}}\widehat{A}$ and so $\bigcap\mathcal{F}\in p$ and thus $\bigcap\mathcal{F}\neq \emptyset$. That is, B has the finite
intersection property, so by Theorem \ref{T1} pick $q\in\mathfrak{m}^\beta$ with $\mathcal{B}\subseteq q$. Then $q\in \bigcap\mathcal{A}$.
\end{proof}
\begin{theorem}
Let $X$ be a nonempty set and let $\mathfrak{m}$ be an infinite $\sigma$-algebra on $X$. Then the sets of the form $\widehat{A}$ are the clopen
subsets of $\mathfrak{m}^\beta$.
\end{theorem}
\begin{proof}
Each set $\widehat{A}$ is closed as well as open.
Suppose that $C$ is any clopen subset of $\mathfrak{m}^\beta$. Let $\mathcal{A}=\{\widehat{A}:A\in \mathfrak{m} \ and \ \widehat{A}\subseteq C \}$.
Since $C$ is open, $\mathcal{A}$ is an open cover of $C$. Since $C$ is closed, it is compact by Theorem 2.5(b). Now by Theorem \ref{T3}, pick a finite subfamily $\mathcal{F}$ of $\mathfrak{m}$ such that $C=\bigcup_{A\in\mathcal{F}}\widehat{A}$. Then by Lemma \ref{L2}, $C=\widehat{\bigcup \mathcal{F}}$.
\end{proof}
\begin{theorem}\label{T5}
Let $X$ be a nonempty set and let $\mathfrak{m}$ be a $\sigma$-algebra on $X$ containing $P_f(X)$.

$(a)$ For every $A\in\mathfrak{m}$, $\widehat{A}=cl_{\mathfrak{m}^\beta}e(A)$ where $e:X\rightarrow\mathfrak{m}^\beta$ is defined by $e(x)=\hat{x}$.

$(b)$ For any $A\in\mathfrak{m}$ and any $p\in \mathfrak{m}^\beta$, $p\in cl_{\mathfrak{m}^\beta}e(A) $ if and only if $A\in p$.
\end{theorem}
\begin{proof}
The proof is obvious.
\end{proof}
\begin{theorem}\label{T4}
Let $X$ be a nonempty set and let $\mathfrak{m}$ be a $\sigma$-algebra on $X$  and $\mathfrak{m}$ separate $X$. Let $e:X\rightarrow\mathfrak{m}^\beta$ is defined by $e(x)=\hat{x}$.

$(a)$ The mapping $e$ is $\mathfrak{m}$-measurable.

$(b)$ If $(X, \mathfrak{m})$ is measurable space and let $P_f(X)\subseteq\mathfrak{m}$, then the mapping $e$ is injective.

 $(c)$ $e(X)$ is a dense subset of $\mathfrak{m}^\beta$.
\end{theorem}
\begin{proof}
$(a)$ Since $A=e^{-1}(\widehat{A})$ for every $A\in \mathfrak{m}$, so the mapping $e$ is $\mathfrak{m}$-measurable.

The proofs $(b)$ and $(c)$ are obvious.
\end{proof}
\begin{theorem}
Let $X$ be a nonempty set and let $\mathfrak{m}$ be a $\sigma$-algebra on $X$ containing $P_f(X)$. If $U$ is an open subset of $\mathfrak{m}^\beta$,
$cl_{\mathfrak{m}^\beta}U $ is also open.
\end{theorem}
\begin{proof}
If $U=\emptyset$, the conclusion is trivial and so we assume that $U\neq\emptyset$. Put
$A=e^{-1}[U]$. We claim first that $U\subseteq cl_{\mathfrak{m}^\beta}e(A)$. So let $p\in U$ and let $\widehat{B}$ be a basic
neighborhood of $p$. Then $U\cap \widehat{B}$ is a nonempty open set and so by Theorem \ref{T4} $(c)$, $U\cap\widehat{B}\cap e(X)\neq \emptyset$.
So pick $b\in B$ with $e(b)\in U$. Then $e(b)\in \widehat{B}\cap e(A)$ and so $\widehat{B}\cap e(A)\neq\emptyset$.
Also $e[A]\subseteq U $ and hence $U\subseteq cl_{\mathfrak{m}^\beta}e(A)\subseteq cl_{\mathfrak{m}^\beta}U$. By Theorem \ref{T5} $(a)$,
$cl_{\mathfrak{m}^\beta}U=cl_{\mathfrak{m}^\beta}e(A)=\widehat{A}$,
and so $cl_{\mathfrak{m}^\beta}U$ is open in $\mathfrak{m}^\beta$.
\end{proof}
\begin{defn}
Let $X$ be a nonempty set, let $\mathfrak{m}$ be a $\sigma$-algebra on $X$ containing $P_f(X)$, and let $\mathcal{A}$ be an $\mathfrak{m}$-filter.
We define $\widehat{\mathcal{A}}=\{p\in \mathfrak{m}^\beta: \mathcal{A}\subseteq p\}$.
\end{defn}
\begin{theorem}
Let $X$ be a nonempty set and let $\mathfrak{m}$ be a $\sigma$-algebra on $X$ containing $P_f(X)$.

$(a)$ If  $\mathcal{A}$ is an $\mathfrak{m}$-filter, then $\widehat{A}$ is a closed subset of $\mathfrak{m}^\beta$.

$(b)$ If $\emptyset \neq A\subseteq \mathfrak{m}^\beta$ and $\mathcal{A}=\bigcap A$, then $\mathcal{A}$ is an $\mathfrak{m}$-filter and $\mathcal{\widehat{A}}=cl_{\mathfrak{m}^\beta}A$.
\end{theorem}
\begin{proof}
$(a)$ Let $p\in \mathfrak{m}^\beta\setminus \mathcal{A} $. Pick $B\in \mathcal{A}\setminus p$. Then $\widehat{X\setminus B}$ is a neighborhood of $p$
which misses $\widehat{\mathcal{A}}$.

$(b)$ $\mathcal{A}$ is the intersection of a set of $\mathfrak{m}$-filters, so $\mathcal{A}$ is an $\mathfrak{m}$-filter. Further, for each $p\in A$,
$\mathcal{A}\subseteq p$ so $A\subseteq \widehat{\mathcal{A}}$ and thus by $(a)$,
$cl_{\mathfrak{m}^\beta}A\subseteq\mathcal{\widehat{A}}$. To see that $\mathcal{\widehat{A}}\subseteq cl_{\mathfrak{m}^\beta}A$,
let $p\in \widehat{\mathcal{A}}$ and let
$B\in p$. Suppose $\widehat{B}\cap A=\emptyset$. Then for each $q\in A$, $X\setminus B\in q$ so $X\setminus B\in \mathcal{A}\subseteq p$,
a contradiction.
\end{proof}
\begin{theorem}
Let $X$ be a nonempty set, let $\mathfrak{m}$ be a $\sigma$-algebra on $X$ containing $P_f(X)$, $p\in \mathfrak{m}^\beta$, and $U\subseteq \mathfrak{m}^\beta $.
If $U$ is a neighborhood of
$p$ in $\mathfrak{m}^\beta $, then $e^{-1}[U]\in p $.
\end{theorem}
\begin{proof}
If $U$ is a neighborhood of $p$, there is a basic open subset $\widehat{A}$ of $\mathfrak{m}^\beta $ for which
$p\in \widehat{A}\subseteq U$. This implies that $A\in p$ and so $e^{-1}[U]\in p $, because $A\subseteq e^{-1}[U]$.
\end{proof}

\section{\textbf{More Topology of $\mathfrak{m}^\beta$ }}

Now let $(X,\mathfrak{m}_1)$ and $(Y,\mathfrak{m}_2)$ be two measurable spaces such that $P_f(X)\subseteq\mathfrak{m}_1$ and $P_f(Y)\subseteq\mathfrak{m}_2$. Also let $\varphi: X\rightarrow Y$ be an
$(\mathfrak{m}_1,\mathfrak{m}_2)$-measurable function. In Lemma 3.1, we show that $\varphi$ has a unique extension.

\begin{lem}\label{L3}
 Let $(X,\mathfrak{m}_1)$ and $(Y,\mathfrak{m}_2)$ be measurable spaces such that  $P_f(X)\subseteq\mathfrak{m}_1$ and $P_f(Y)\subseteq\mathfrak{m}_2$. Let $\varphi: (X,\mathfrak{m}_1)\rightarrow (Y,\mathfrak{m}_2)$ be an
$(\mathfrak{m}_1,\mathfrak{m}_2)$-measurable function. Then there exists a continuous function
$\varphi_\beta:\mathfrak{m}_1^\beta\rightarrow\mathfrak{m}_2^\beta$ such that

$(a)$ $\varphi_\beta(p)=\{A\in \mathfrak{m}_2: \varphi^{-1}(A)\in p \}$.

$(b)$ $e_Y\circ\varphi$ = $\varphi_\beta\circ e_X$.

$(c)$ If $A\in p\in \mathfrak{m}_1^\beta$, then $\varphi(A)\in \varphi_\beta(p)$ and if $B\in \varphi_\beta(p)$, then $\varphi^{-1}(B)\in p$.
\end{lem}
\begin{proof}
$(a)$ It is obvious that $\{A\in \mathfrak{m}_2: \varphi^{-1}(A)\in p \}$ is an $\mathfrak{m}_2$-ultrafilter. For every $p\in \mathfrak{m}_1^\beta$, let
\[
g(p)=\{A\in \mathfrak{m}_2: \varphi^{-1}(A)\in p \}.
\]
Then for every $x\in X$,
\begin{align*}
g(e(x)) &=\{A\in \mathfrak{m}_2:\varphi^{-1}(A)\in e(x)\} \\
&= \{A\in \mathfrak{m}_2:x\in \widehat{\varphi^{-1}(A)}\} \\
&= \{A\in \mathfrak{m}_2:\varphi(x)\in A\}
\end{align*}
So $g(e(x))=\varphi(x)$.

To see that $g$ is continuous, let $A\in \mathfrak{m}_2$. Then
 $\widehat{(g\circ e)^{-1}(A)}=\widehat{\varphi^{-1}(A)}$.
 Since $g$ is a continuous extension
of $\varphi$, we have $\widetilde{\varphi}=g$.

The proofs $(b)$ and $(c)$ are obvious.
\end{proof}
\begin{theorem}
Let $(X,\mathfrak{m})$ be a measurable space, such that $P_f(X)\subseteq \mathfrak{m}$.
Then every $G_\delta$-subset of $X^*$ has nonempty interior in $X^*$.
\end{theorem}
\begin{proof}
Choose $p\in\bigcap_{n\in \mathbb{N}}(U_n\cap X^*)$, where $U_n$ is an open subset of
$\mathfrak{m}^\beta$ for each $n\in \mathbb{N}$. Choose a subset $A_n$ of $X$ such that
$p\in A^*_n\subseteq U_n$ and $A_{n+1}\subseteq A_n$ for each $n\in \mathbb{N}$. Now choose an infinite sequence
$\{a_n\}_{n=1}^\infty$ of distinct points of $X$ such that $a_n\in A_n$. Let $A=\{a_n:n\in \mathbb{N}\}$.
It is obvious that $A^*$ is an open subset of $X^*$ and
$A^*\subseteq U_n$ for each $n\in \mathbb{N}$.
\end{proof}
\begin{theorem}
Let $(X,\mathfrak{m})$ be a measurable space, such that $P_f(X)\subseteq \mathfrak{m}$.
Then every countable union of nowhere dense subsets of $X^*$ is nowhere dense in $X^*$.
\end{theorem}
\begin{proof}
Let $A_n$ be a nowhere  dense subset of $X^*$ for each $n\in \mathbb{N}$. We show that
$B=\bigcup_{n\in \mathbb{N}}clA_n$ is nowhere dense. Suppose instead that $U=int_{X^*}clB\neq\emptyset$. By
the Baire Category Theorem $X^*\setminus B$ is dense in $X^*$ so $U\setminus B\neq\emptyset$. Thus $U\setminus B$ is a
nonempty $G_\delta$-set which thus by Theorem 3.2, has nonempty interior. This is a contradiction.
\end{proof}
\section{\textbf{ Measurable Semigroups}}
In this section we extend the operation of a measurable semigroup.
\begin{defn}
(a) A right measurable semigroup is a triple $(S,\cdot,\mathfrak{m})$ where $(S,\cdot)$ is a semigroup, $(S,\mathfrak{m})$ is a measurable space, and $\rho_x:S\rightarrow S$ is $(\mathfrak{m},\mathfrak{m})$-measurable function for each $x\in S$ where $\rho_x(y)=yx$ for each $y\in S$.

(b) A left measurable semigroup is a triple $(S,\cdot,\mathfrak{m})$ where $(S,\cdot)$ is a semigroup, $(S,\mathfrak{m})$ is a measurable space, and $\lambda_x:S\rightarrow S$ is $(\mathfrak{m},\mathfrak{m})$-measurable function for each $x\in S$ where $\lambda_x(y)=xy$ for each $y\in S$.

(c) A measurable semigroup is a triple $(S,\cdot,\mathfrak{m})$ where $(S,\cdot)$ is a semigroup, $(S,\mathfrak{m})$ is a measurable space, and $\pi:S\times S\rightarrow S$ is $(\mathfrak{m}\otimes\mathfrak{m},\mathfrak{m})$-measurable function where $\pi(x,y)=xy$.
\end{defn}
In this section we assume that $P_f(S)\subseteq \mathfrak{m}$, so $\{s\}$ belongs to $\mathfrak{m}$ for each $s\in S$.
\begin{defn}
Let $(X, \mathfrak{m})$ be a measurable space.
The tensor product of two $\mathfrak{m}$-ultrafilters $p$ and $q$ in $\mathfrak{m}^\beta$ is denoted by $p\otimes q$ and is defined as following
\[
p\otimes q=\{A\in \mathfrak{m}\otimes \mathfrak{m}:\{s\in X:A_s\in q\}\in p\}
\]
\end{defn}
where $A_s=\{t\in S:(s,t)\in A\}$.
\begin{lem}
If $p,q\in \mathfrak{m}^\beta$, then $p\otimes q\in (\mathfrak{m}\otimes\mathfrak{m})^\beta$.
\end{lem}
\begin{proof}
The proof is obvious.
\end{proof}
\begin{lem}
The following statements hold:

$(a)$ $S\otimes S$ is dense in $(\mathfrak{m}\otimes\mathfrak{m})^\beta$.

$(b)$ $\lim_{s\rightarrow p}\lim_{t\rightarrow q}s\otimes t=p\otimes q$ for every $p,q\in \mathfrak{m}^\beta$.

\end{lem}
\begin{proof}
The proof is obvious.
\end{proof}
\begin{lem}
The mapping $R_q:S\rightarrow(\mathfrak{m}\otimes\mathfrak{m})^\beta$ by $R_q(s)=s\otimes q$ is a
$(\mathfrak{m},\mathfrak{m}\otimes \mathfrak{m})$-measurable function where $q\in \mathfrak{m}^\beta$.
\end{lem}
\begin{proof}
Since $\mathfrak{m}\times \mathfrak{m}$ generates $\mathfrak{m}\otimes \mathfrak{m}$, then

\begin{align*}
R_q^{-1}(A\times B) &= \{s\in S:s\otimes q\in\widehat{A\times B}\} \\
&=  \{s\in S:A\times B\in s\otimes q\}\\
&=\{s\in S:\{u:\{t:(u,t)\in A\times B\}\in q\}\in s\}\\
&=  \{s\in S:\{t:(s,t)\in A\times B\}\in q\}\\
&=  \{s\in S:s\in A, B\in q\}=A\in\mathfrak{m}\\
\end{align*}
where $A,B\in \mathfrak{m}$.
\end{proof}
\begin{lem}
Let $s\in S$ and $q\in \mathfrak{m}^\beta$. Then $\pi^\beta\circ R_q(s)=\pi^\beta(s\otimes q)=s\ast q$
\end{lem}
\begin{proof}
The proof is obvious.
\end{proof}
\begin{theorem}
Let $(S, \cdot, \mathfrak{m})$ be a measurable semigroup. Then there is a unique binary
operation $\ast:\mathfrak{m}^\beta\times\mathfrak{m}^\beta\rightarrow\mathfrak{m}^\beta$ satisfying the following three
conditions:

$(a)$ For every $s,t\in S$, $s\ast t=s\cdot t$.

$(b)$ For each $q\in\mathfrak{m}^\beta$, the function $\rho_q:\mathfrak{m}^\beta\rightarrow\mathfrak{m}^\beta$ is continuous,
where $\rho_q(p)=p\ast q$.

$(c)$ For each $s\in S$, the function $\lambda_s:\mathfrak{m}^\beta\rightarrow\mathfrak{m}^\beta$ is continuous,
where $\lambda_s(q)=s\ast q$.
\end{theorem}
\begin{proof}
Given any $s\in S$, define
$l_s:S\rightarrow S$ by $l_s(t)=st$, then $l_s$ is a measurable function. So there is a unique continuous function $\lambda_s:\mathfrak{m}^\beta\rightarrow \mathfrak{m}^\beta$ such that $\lambda_s(e(t))=l_s(t)$ for each $t\in S$. If $s \in S$ and $q\in \mathfrak{m}^\beta$, we define
$s\ast q=\lambda_s(q)$. Then $(c)$ holds and so does $(a)$, because  $\lambda_s$ extends $l_s$.

Now we extend $\ast$ to the rest of $\mathfrak{m}^\beta\times\mathfrak{m}^\beta$. Given $q\in \mathfrak{m}^\beta$, define
$r_q=\pi^\beta\circ R_q:S\rightarrow\mathfrak{m}^\beta$. The mapping $r_q$ is measurable, so there is a unique continuous extension $\rho_q:\mathfrak{m}^\beta\rightarrow \mathfrak{m}^\beta$. If $p,q\in \mathfrak{m}^\beta$, we define
$p\ast q=\rho_q(p)$. This is the only possible definition
which satisfies the required conditions.
\end{proof}
\begin{theorem}
Let $(S, \cdot, \mathfrak{m})$ be a measurable semigroup. Then the extended operation on $\mathfrak{m}^\beta$ is associative.
\end{theorem}
\begin{proof}
The proof is obvious.
\end{proof}
\begin{theorem}
Let $(S, \cdot, \mathfrak{m})$ be a measurable semigroup. Then $(\mathfrak{m}^\beta,\ast)$ is a compact right semigroup.
\end{theorem}
\begin{proof}
The proof is obvious.
\end{proof}

\begin{theorem}
Let $(S, \cdot, \mathfrak{m})$ be a measurable semigroup and $A\in\mathfrak{m}$.

$(a)$ For any $s\in S$ and $q\in \mathfrak{m}^\beta$, $A\in s\cdot q$ if and only if $s^{-1}A\in q$.

$(b)$ For any $p,q\in\mathfrak{m}^\beta$, $A\in p\cdot q$ if and only if $\{s\in S:s^{-1}A\in q\}\in p$.
\end{theorem}
\begin{proof}
$(a)$ Necessity. Let $A\in s\cdot q$. Then $A\in\lambda_s(q)$ and hence by Lemma \ref{L3}, $s^{-1}A=\lambda_s^{-1}(A)\in q$.

Sufficiency. Assume $s^{-1}A\in q$ and suppose that $A\notin s\cdot q$. Then $S\setminus A\in s\cdot q $.
So, by the already established necessity, $s^{-1}(S\setminus A)\in q$. This is a contradiction since
$s^{-1}A\cap s^{-1}(S\setminus A)=\emptyset$.

$(b)$ Necessity. Let $A\in p\cdot q$. Then by Lemma \ref{L3},
\[
\{s\in S:\rho_q(s)\in \widehat{A}\}=\rho_q^{-1}(A)\in p.
\]
So $\{s\in S:A\in s\cdot q\}\in p$. Hence  $\{s\in S:s^{-1}A\in q\}\in p$.

Sufficiency. Let $\{s\in S:s^{-1}A\in q\}\in p$. Since
\begin{align*}
 \{s\in S:s^{-1}A\in q\} &=\{s\in S:A\in s\cdot q\} \\
&=  \{s\in S:A\in \rho_q(s)\}\\
&=\{s\in S:s\in \rho_q^{-1}(A)\}\\
\end{align*}
 $\rho_q^{-1}(A)\in p$. So $A\in p\cdot q$.
\end{proof}
\begin{defn}
Let $(S, \cdot, \mathfrak{m})$ be a measurable semigroup, then $\mathfrak{m}^\beta$ is called measurable semigroup compactification of
$(S, \cdot, \mathfrak{m})$.
\end{defn}
Now we state some algebraic properties of semigroup extension of measurable semigroup $(S, \cdot, \mathfrak{m})$.
\begin{theorem}
Let $(S,\cdot,\mathfrak{m})$ be a measurable semigroup, then $\mathfrak{m}^\beta$ has an idempotent.
\end{theorem}
\begin{proof}
See Theorem 2.5 in \cite{hindbook}.
\end{proof}
\begin{exam}
a) Let $\LL$ denote the collection of all Lebesgue measurable subsets od real numbers. Then
$(\LL^\beta,+)$ and $(\LL^\beta,\cdot)$ are measurable semigroup compactification of $(\mathbb{R},\LL)$.

b)Every  vector  space  has  a  Hamel  basis, i.e. a maximal linearly independent subset. Since Real numbers as a vector space on rational numbers has infinite dimension, therefore there exists an uncountable maximal linearly independent $I\subseteq [0,1]$. Now choice a countable subset $A\subseteq I$. Then the generated vector space by $I\setminus A$ is not closed. Define $i(x)=x$ when $x\in A$ and $i(x)=0$ for $x\in I\setminus A$. We denote $i$ as to extension $i$ to $\mathbb{R}$ as a linear transformation. It is obvious that $i$ is Borel measurable but it is not continuous.

Now define $*_i:\mathbb{R}\times \mathbb{R}\rightarrow\mathbb{R}$ by $*_i(x,y)=i(x+y)$ for each $x,y\in \mathbb{R}$.
Then $(\mathbb{R},*_i)$ is measurable semigroup, but it is not semitopological semigroup.

\end{exam}
\begin{theorem}
Let $(S, \cdot, \mathfrak{m})$ be a measurable semigroup. Then $S^*=\mathfrak{m}^\beta\setminus S$ is a
subsemigroup of $\mathfrak{m}^\beta$ if and only
if for any $A\in P_f(S)$ and for any infinite subset $B$ of $\mathfrak{m}$ there exists $F\in P_f(B)$ such
that $\bigcap_{x\in F}x^{-1}A$ is finite.
\end{theorem}
\begin{proof}
Necessity. Let $A\in P_f(S)$ and an infinite subset $B\in \mathfrak{m}$ be
given. Suppose that for each $F\in P_f(B)$, $\bigcap_{x\in F}x^{-1}A$ is infinite. Then $\{x^{-1}A:x\in B\}$
has the property that all of its finite intersections are infinite so by Corollary \ref{C1}
we may pick $p\in S^*$ such that $\{x^{-1}A:x\in B\}\subseteq p$. Pick $q\in S^*$ such that $B\in q$.
Then $A\in q\cdot p$ and $A$ is finite so $q\cdot p\in S$, a contradiction.

Sufficiency. Let $p,q\in S^*$ be given and suppose that $q\cdot p=y\in S$, (that is,
precisely, that $q\cdot p$ is the $\mathfrak{m}$-principal ultrafilter generated by y). Let $A=\{y\}$ and let
 $B=\{x\in S:x^{-1}A\in p\}$. Then $B\in p$ while for each $F\in P_f(B)$, one has
$\bigcap_{x\in F}x^{-1}A\in p$ so that $\bigcap_{x\in F}x^{-1}A$ is infinite, a contradiction.
\end{proof}
\begin{defn}
Let $S$ be a measurable semigroup. We say that
S is weakly left cancellative if and only if $\lambda_s^{-1}(\{y\})$ is finite for every $x,y \in S$.
\end{defn}
\begin{theorem}
Let $(S, \cdot, \mathfrak{m})$ be a measurable semigroup. Then $S^*$ is a left ideal of $\mathfrak{m}^\beta$ if and
only if $S$ is weakly left cancellative.
\end{theorem}
\begin{proof}
Necessity. Let $x,y\in S$ be given, let $A=\lambda_s^{-1}(\{y\})$ and suppose that $A$ is
infinite. Pick $p\in S^*\cap \widehat{A}$. Then $y\cdot p=x$, a contradiction.

Sufficiency. Since $S$ is infinite, $S^*\neq\emptyset$. Let $p\in S^*$, let $q\in \mathfrak{m}^\beta$ and suppose
that $q\cdot p= x\in S$. Then $\{x\}\in q\cdot p$ so $\{y\in S:y^{-1}\{x\}\in p\}\in q$ and is
hence nonempty. So pick $y\in S$ such that $y^{-1}\{x\}\in p$. But $y^{-1}\{x\}=\lambda_y^{-1}(\{x\})$ so
$\lambda_y^{-1}(\{x\})$ is infinite, a contradiction.
\end{proof}
\begin{theorem}
Let $(S, \cdot, \mathfrak{m})$ be a measurable semigroup. The following statements are equivalent:

$(a)$ $S^*$ is a right ideal of $\mathfrak{m}^\beta$.

$(b)$ Given any finite subset $A\in\mathfrak{m}$, any sequence $\langle z_n\rangle_{n=1}^{\infty}$ in $S$, and any one-to-one
sequence $\langle x_n\rangle_{n=1}^{\infty}$ in $S$, there exist $n< m$ in $\mathbb{N}$ such that $x_n\cdot z_m\notin A$.

$(c)$ Given any $a\in S$, any sequence $\langle z_n\rangle_{n=1}^{\infty}$ in $S$, and any one-to-one sequence
$\langle x_n\rangle_{n=1}^{\infty}$ in $S$, there exist $n< m$ in $\mathbb{N}$ such that $x_n\cdot z_m \neq a$.
\end{theorem}
\begin{proof}
See Theorem 4.32 in \cite{hindbook}.
\end{proof}

\section{\textbf{Applications}}
Recall a $\sigma$-algebra on a topological space $X$ is the Borel $\sigma$-algebra
generated by all open sets; it is denoted by $\B(X)$. A set in $\B(X)$ is called the Borel set in the space X.

A $\sigma$-algebra on a topological space X is generated by all sets of the form
\[
\{x\in X:f(x)>0\}
\]
where $f$ is a real valued continuous function on X, is called Baire
$\sigma$-algebra and denoted by $Ba(X)$. So $A\in Ba(X)$ is called a Baire set. By Corollary 5.3.5 in \cite{Bogachev},
when $X$ is a metric space then
$A\subseteq X$ is Borel set if and only if $A$ is a Baire set.

Let $\LL$ denote the collection of all Lebesgue measurable sets on $S=(0,+\infty)$, and let $\lambda$ be the Lebesgue measure on real numbers.
It is obvious that $\B(S)\subseteq \LL$. For each $p,q\in\LL^\beta$, $p+q$ and $p\cdot q$ are well define, see Theorem 4.10. So $\LL^\beta$ is a multiplicative and additive semigroup.

Now define
\[
\LL_\circ=\{A\in \B:\lambda(A)=0\}
\]
and for $k\in (0,1]$ define
\[
\LL_k=\{A\in \B:\lambda(A)>k\}.
\]
We define $\LL_+=\bigcup_{k>\circ}\LL_k$. It is obvious that $\LL_0$, $\LL_k$ and $\LL_k$ are partition regular,
and also $\LL_0$, $\LL_k$ and $\LL_+$ are invariant under translate, where $A+x$ is
the translate of $A$ by $x\in S$. Therefore the below collections
\[
\mathcal{L}_\circ^\beta=\{p\in \mathcal{L}^\beta:\exists A\in p\,\,\,\,\lambda(A)=0\},
\]
\[
\mathcal{L}_k^\beta=\{p\in \mathcal{L}^\beta:\forall A\in p\,\,\,\,\lambda(A)>k\}
\]
and
\[
\mathcal{L}_+^\beta=\{p\in \mathcal{L}^\beta:\forall A\in p\,\,\,\,\lambda(A)>0\}
\]
are non-empty sets, by Theorem 2.9.

Let $\mathfrak{m}$ be a $\sigma$-algebra on $X$ and let $A\in\mathfrak{m}$, then $\mathfrak{m}_A=\{T\cap A:T\in \mathfrak{m}\}$
 is a $\sigma$-algebra and we can write $\mathfrak{m}_A^\beta\subseteq \mathfrak{m}^\beta$, because
for each $p_A\in\mathfrak{m}_A^\beta$ there exists a unique $p\in \mathfrak{m}$ such that $p_A\subseteq p$.
\begin{lem}
a) $\mathcal{L}_k^\beta$ and $\mathcal{L}_+^\beta$ are left ideals of $(\mathcal{L}^\beta,+)$. Also $\mathcal{L}_+^\beta$ is left ideal of $(\mathcal{L}^\beta,\cdot)$. \\
b) $\mathcal{L}_\circ^\beta$ is a multiplicative and additive subsemigroup of $\mathcal{L}^\beta$.
\end{lem}
\begin{proof}
a) Pick $p\in\mathcal{L}_k^\beta$ and $q\in\LL^\beta$. Let $A\in q+p$, so
\[
\{x\in S:-x+A\in p\}\in q.
\]
Since $p\in\mathcal{L}_k^\beta$ so $\lambda(A)=\lambda(-x+A)>k$. This implies $q+p\in\mathcal{L}_k^\beta$.
By similar way, $\mathcal{L}_+^\beta$ is left ideal of $(\mathcal{L}^\beta,+)$.

b) Pick $p,q\in \mathcal{L}_\circ^\beta$, so there exist $A_p\in p$ and $A_q\in q$
such that $A_p\cap A_q=\emptyset$ and $\lambda(A_p)=\lambda(A_q)=0$. Now Let $A=A_p\cup A_q$, and define
$p_\circ=\{T\cap A:T\in p\}$ and $q_\circ=\{T\cap A:T\in q\}$. It is obvious that $p_\circ$ and $q_\circ$
are $\LL$-ultrafilters on $A$, $p_\circ\subseteq p$ and $q_\circ\subseteq q$. It is obvious that
$p_\circ+q_\circ\subseteq p+q$. By similar way, we can show that $\mathcal{L}_\circ^\beta$ is a multiplicative subsemigroup.
\end{proof}
For $x\in S$, we define
$$x^*= \{p\in \LL^\beta:x\in \bigcap_{A\in p}cl_{S}A\}.$$
It is obvious that $x^*=\{p\in \LL^\beta:\forall y>0,\,\,\,(x-y,x+y)\in p\}$, and $\widehat{x}\in x^*$. We say $p\in x^*$ is
a near point to $x$. We define $B(S)= \cup_{x\in S}x^*$  and $\infty^*=\LL^\beta-B(S)$. An element $p\in B(S)$ is called a bounded ultrafilter and $p\in \infty^*$ is called an unbounded ultrafilter. It is obvious that
\[
\infty^*=\{p\in \LL^\beta:\forall x>0,\,\,\,(x,+\infty)\in p\}.
\]
\begin{lem}
For each $x,y\in S$, then\\
a) $x^*$ is a non empty and compact subset of $\LL^\beta$,\\
b) $x^*+y^*\subseteq (x+y)^*$, and \\
c) $x^*y^*\subseteq (xy)^*$.
\end{lem}
\begin{proof}
See Lemma 2.3(i) in \cite{va}.
\end{proof}
By above Lemma, it is obvious that $0^*$ is a compact multiplicative and additive subsemigroup of $\LL^\beta$.
\begin{lem}
$\infty^*$ is a multiplicative and additive subsemigroup of $\LL^\beta$.
\end{lem}
\begin{proof}
 Let $p,q\in \infty^*$ and let $pq\in x^*$ for some $x\in S$. Therefore $(a,b)\in pq$ for some $a,b\in S$, and so for some $t\in S$,  $t^{-1}(a,b)\in q$ is a contradiction. So $\infty^*$ is a multiplicative subsemigroup. By similar way, $\infty^*$ is an additive semigroup.
\end{proof}
We say that $A\subseteq S$ is meager if and only if $A$ is the countable union of
nowhere dense sets. Now we state a relation between Baire and meager sets.

\begin{defn}
a) A Lebesgue measurable set $A\subseteq S$ is called Baire large at $\infty$ if and only if
for each $x>0$, $A\cap (x,+\infty)$ is not meager.
The collection of all Baire large sets at $\infty$ is denoted by $BL(\infty)$.

b)  A Lebesgue measurable set $A\subseteq S$ is called Baire large at $0$ if and only if for each $x>0$, $A\cap (0,x)$ is not meager. The collection of
all Baire large sets at $0$ is denoted by $BL(0)$.
\end{defn}
\begin{lem}
a) $BL(\infty)$ and $BL(0)$ are partition regular.\\
b) $\B\LL(0)=\{p\in\LL^\beta: \forall A\in p, A\in BL(0)\}$ is left ideal in $(0^*,\cdot)$.\\
c) $\B\LL(\infty)=\{p\in\LL^\beta: \forall A\in p, A\in BL(\infty)\}$ is left ideal in $(\infty^*,\cdot)$.\\
d) $\B\LL(\infty)$ is left ideal in $(\infty^*,+)$.
\end{lem}
\begin{proof}
a) It is obvious.\\
b) By a) and Theorem 2.9 implies $BL(0)$ is non empty. Now similar Lemma 17.39 in \cite{hindbook}, let $p\in BL(0)$, $q\in 0^*$ and let
$A\in qp$. Pick $x\in (0,1)$ such that $x^{-1}A\in p$. Let $x^{-1}A\cap (0,\epsilon)$ is not meager for some $\epsilon>0$, so
$x(x^{-1}A\cap (0,\epsilon))$ is not meager and $x(x^{-1}A\cap (0,\epsilon))\subseteq A\cap (0,\epsilon)$ .

c) and d). The proof is similar to b).
\end{proof}
\begin{defn}
a)  $A\subseteq (0,+\infty)$ is  called syndetic if and only if there exists some $G\in P_f(0,+\infty)$ such that
\[
(0,+\infty)=\bigcup_{t\in G}-t+A.
\]
b) $A\subseteq (0,+\infty)$ is  called thick if and only if for each $F\in P_f((0,+\infty))$ there
exists  $x\in A$ such that $x+F\subseteq A$.

c) $A\subseteq (0,+\infty)$ is  called piecewise syndetic if and only if $\bigcup_{t\in G}-t+A$ is thick for some $G\in P_f(0,+\infty)$.
\end{defn}
\begin{theorem}
a) Let $p\in\LL^\beta$. Then $p\in K(\LL^\beta,+)$ if and only if $\{x>0:-x+A\in p\}$ is syndetic for each $A\in p$.

b) Let $A\subseteq (0,+\infty)$. Then $K(\LL^\beta,+)\cap \overline{A}\neq\emptyset$ if and only if $A$ is piecewise syndetic.
\end{theorem}
\begin{proof}
The proof is similar to the proof of Theorem 4.39 and 4.40 in \cite{hindbook}.
\end{proof}
The following Theorem has been stated as Theorem 1.3 in \cite{Lisan}.
\begin{theorem}
The closure of the minimal ideal $K(\LL^\beta,+)$ is a left ideal of $(\LL^\beta,\cdot)$. In particular there is
a multiplicative idempotent in $cl_{\LL^\beta}K(\LL^\beta)$.
\end{theorem}
\begin{proof}
Let $p\in cl_{\LL^\beta}K(\LL^\beta,+)$ and $q\in\LL^\beta$. We show that $qp\in cl_{\LL^\beta}K(\LL^\beta,+)$. So let $A\in qp$.
Then $\{x>0:x^{-1}A\in p\}\in q$. Pick $x>0$ such that $x^{-1}A\in p$. Since $p\in cl_{\LL^\beta}K(\LL^\beta)$, so $x^{-1}A$ is piecewise syndetic.
Since $\overline{x^{-1}A}$ is a neighborhood of $p$, so there exists an $\eta\in\overline{x^{-1}A}\cap K(\LL^\beta,+)$. This implies that
$x^{-1}A\in\eta$, and hence $A\in x\eta$. Now Let $l_x:(0,\infty)\rightarrow (0,+\infty)$ is defined by $l_x(y)=xy$. It is obvious that $l_x$ is an additive isomorphism.  So $l_x$ has a unique continuous isomorphism extension $l^\beta_x:\LL^\beta\rightarrow\LL^\beta$. Therefore $xK(\LL^\beta,+)=K(\LL^\beta,+)$ and so $A\in x\eta$ is piecewise syndetic. This implies that $\overline{A}\cap K(\LL^\beta,+)\neq\emptyset$ for each $A\in qp$. So $qp\in cl_{\LL^\beta}K(\LL^\beta,+)$.
\end{proof}
\bibliographystyle{alpha}

\end{document}